\documentclass[12pt,a4paper,oneside]{article}
\usepackage{amsmath,amsthm}
\usepackage{amssymb}
\usepackage{amsfonts}
\usepackage{graphicx}
\usepackage[latin1]{inputenc}
\usepackage{latexsym}
\usepackage{wrapfig}
\usepackage{makeidx}

\usepackage{hyperref}
\usepackage[all]{xy}
\usepackage{enumerate}
\usepackage{enumitem}

\title{\sc Holomorphic functions with  large cluster sets}
\author{Thiago R. Alves\thanks{This study was financed in part by the Coordenação de Aperfeiçoamento de Pessoal de Nível Superior - Brasil (CAPES) - Finance Code 001 and FAPEAM.}~ and Daniel Carando\thanks{Supported by CONICET-PIP 11220130100329CO and ANPCyT PICT 2015-2299.}}
\date{}
\hypersetup{colorlinks,%
citecolor=black,%
filecolor=black,%
linkcolor=black,%
urlcolor=black,%
pdftex}

\newtheorem{thm}{Theorem}[section]

\newtheorem{lem}[thm]{Lemma}
\newtheorem{pps}[thm]{Proposition}

\theoremstyle{definition}
\newtheorem{dfn}[thm]{Definition}
\newtheorem{rem}[thm]{Remark}

\begin{document}

\maketitle

\begin{abstract}
We study linear and algebraic structures in sets of bounded holomorphic functions on the ball which have large cluster sets at every possible point (i.e., every point on the sphere in several complex variables and every point of the closed unit ball of the bidual in the infinite dimensional case). We show that this set is  strongly $\mathfrak{c}$-algebrable for all separable Banach spaces. For specific spaces including $\ell_p$ or duals of Lorentz sequence spaces, we have strongly $\mathfrak{c}$-algebrability and spaceability even for the subalgebra of uniformly continous holomorphic functions on the ball.\end{abstract}

\renewcommand{\thefootnote}{\fnsymbol{footnote}}
\footnotetext{\emph{Key words and phrases.} Holomorphic functions. Cluster sets. Lineability. Algebrability. Spaceability. Banach spaces.}
\footnotetext{\emph{2010 Mathematics Subject Classification.}  46J15, 32A40, 46G20.}

\section{Introduction and main results}

There is an increasing interest in the search  linear or algebraic structures in sets of functions with special (usually \emph{bad}) non-linear  properties. A seminal example of this kind of results is the construction in \cite{Gu91} of infinite dimensional subspaces of $C([0,1])$  containing only of nowhere differentiable functions (except the zero function). Since then, many efforts were devoted in this direction, especially in the last years. We refer the reader to \cite{AronBerPelSeo} for a complete monograph in the subject (see also \cite{AronMaestro, BerPelSeo}). In the setting of complex holomorphic functions, several authors considered linear and algebraic structures in sets of non-extendible functions (see \cite{Alves1,Ber2} and the references therein). Other distinguished properties of holomorphic functions were studied, for example, in  \cite{AlvBot18} or \cite{Lo-Sa}.

In this work, we study linear and algebraic structures in the set of holomorphic functions with \emph{large cluster sets at every point}, both for finite and infinitely many variables.
Cluster values and cluster sets of holomorphic functions in the complex disk $\mathbb D$ were first considered by I.~J.~Schark (a fictitious name chosen by eight brilliant mathematicians of the time) in~\cite{Schark61}. Their motivation was to relate the set of cluster values of a bounded function $f$ at a point in the unit circle $S$ with the set of evaluations $\varphi(f)$ of elements $\varphi$ in the spectrum of the algebra $H^\infty$ over that point. Different authors have studied the analogous problem in the infinite dimensional setting \cite{Aron4,Aron1,JoOr15,Ortega-Prieto18}.
Let us remark that for a bounded holomorphic function $f$ on $\mathbb D$, a large cluster set of $f$ at some $z_0\in  S$  means that $f$ has a wild behaviour as $z\to z_0$ (the cluster set consists of all limit values of $f(z)$ as $z\to z_0$). So, in the one dimensional case,  we are interested in those functions that have this wild behaviour at every point of the unit circle. This is our non-linear property, which can be rated as bad, in opposition to continuity at the boundary, which plays the r\^ole of the good property. Since a cluster set is a compact connected subset of $\mathbb C$, it is considered large whenever it contains a disk.

Let us begin in the context of several complex variables. We consider a norm $\|\cdot\|$ in $\mathbb C^n$ and the corresponding finite dimensional Banach space $E=(\mathbb C^n,\|\,\|)$. We write $B$ and $S$ for the open unit ball and the unit sphere of $E$, respectively. Let $\mathcal{H}^\infty(B)$ denote the algebra of all bounded holomorphic functions on $B$.
The cluster set of a function $f\in \mathcal{H}^\infty(B)$ at a point $z\in \overline B$ is the set $Cl(f,z)$ of all limits of values of $f$ along sequences converging to $z$. For $z$ in the open unit ball this cluster set contains just one point: $f(z)$; but for $z\in S$ the situation can be very different.

Our first main theorem is the following.

\begin{thm}\label{main-findim}
For $E=(\mathbb C^n,\|\,\|)$, the set of functions $f\in \mathcal{H}^\infty(B)$ such that there exists a (fixed) disk centered at the origin which is contained in $Cl(f,z)$ for every $z\in S$ is strongly $\mathfrak{c}$-algebrable.
\end{thm}

We remark that the functions considered in the previous theorem satisfy more than having a disk in each cluster set: there exists a fixed disk which is contained in every cluster set. The same happens with Theorem~\ref{main-infdim} below.

Now we consider an infinite dimensional complex Banach space $E$.  The symbol $B_E$ (or $B$ if there is no ambiguity) represents the open unit ball of $E$, while $S_E$ (or $S$) represents the unit sphere. Also, we write  $B^{**} := B_{E^{**}}$ and $\overline{B}^{**}  := \overline{B_{E^{**}}}$, where $E^{**}$ denotes the topological bidual of $E$.

In this case, for $f\in \mathcal{H}^\infty(B)$ and  $z \in \overline{B}^{**}$, the {cluster set of $f$
at $z$} is the set $Cl(f,z)$ of all limits of values of $f$ along nets in $B$ weak-star converging to $z$. More precisely,
\begin{equation*}
\begin{split}
 Cl(f,z)=\{\lambda\in\mathbb C\,:\, \textrm{there exists a net } & (x_\alpha)\subset B \\
 & \textrm{ such that }x_\alpha\overset{w(E^{**},E^*)}{\longrightarrow}z,\textrm{ and } f(x_\alpha)\to\lambda\}.
\end{split}
\end{equation*}

Note that in the finite dimensional case this definition coincides with the previous one, since convergence and weak-star convergence are equivalent.
The cluster set  $Cl(f,z_0)$ can also be seen as the intersection of the closures of $f(U\cap
B)$, where $U$ ranges over  the weak-star neighborhoods of $z_0$ (or over some basis of weak-star neighborhoods). With this, it can be seen that the cluster set $Cl(f,z_0)$, $z_0\in \overline
B^{**}$, is a nonempty compact connected set.

In the infinite dimensional case, the cluster set can be large even at points in the interior of the ball. In fact, we have the following.

\begin{thm}\label{main-infdim}
If $E$ is a separable infinite-dimensional Banach space, then
the set of functions $f\in \mathcal{H}^\infty(B)$ such that there exists a (fixed) disk centered at the origin which is contained in $Cl(f,z)$ for every $z\in \overline{B}^{**}$  is strongly $\mathfrak{c}$-algebrable.
\end{thm}

In fact, if $E$ is not separable, the same holds for the set of functions whose cluster set at every $z\in A$ contains a common disk centered at 0, for any $A\subset \overline{B}^{**}$ which is weak-star separable.

\medskip

Recall that $A_u(B)$ is the
Banach algebra of all uniformly continuous holomorphic functions on the unit ball
$B$. Equivalently, $A_u(B)$ is the closure in  $\mathcal{H}^\infty(B)$ of continuous polynomials. As a consequence of \cite[Corollary 2.5]{Aron4}, functions in $A_u(B_{\ell_p})$ have trivial cluster sets at points of $S_{\ell_p}$ for $1\le p<\infty$.
Moreover, as we see in Proposition~\ref{notinAu}, a function $f\in \mathcal{H}^\infty(B_{\ell_p})$ for which there exists a fixed disk contained in $Cl(f,z)$ for every $z\in {B_{\ell_p}}$ cannot belong to $A_u(B_{\ell_p})$ ($1\le p < \infty$). So we do not expect a result like Theorem~\ref{main-infdim} to hold for $A_u(B_{\ell_p})$. The same happens for some duals/preduals Lorentz sequence spaces (see Section~\ref{sec-elep} for the definitions). However, if we only ask  cluster sets at $z\in {B}$ to contain disks (whose radii depend on the point), we have both strongly  $\mathfrak{c}$-algebrability and spaceability.

\begin{thm}\label{main-ell_p} Let $E$ be either $\ell_p$ ($1\le p <\infty$) or $d(w,p)^*$ ($1< p <\infty$) or $d_*(w,1)$ with $w\in \ell_s$ for some $1<s<\infty$.
Then, the set of functions $f\in A_u(B_{E})$ whose cluster set at every $x\in B$  contains a disk is strongly  $\mathfrak{c}$-algebrable and contains {(up to the zero function)} an isometric copy of $\ell_\infty$. In particular, it is spaceable.
\end{thm}

We remark that the  copy of $\ell_\infty$ obtained in the previous theorem is actually contained in the subspace of $m$-homogeneous polynomials, where $m \ge p$ for $\ell_p$ ($1<p<\infty$), $m\ge p'$ for $d(w,p)^*$ ($1< p <\infty$), $m\ge s'$ for $d_*(w,1)$ and $m$ can be any even number for $\ell_1$.

Finally, we state the following spaceability result for the case $E=c_0$.

\begin{thm}\label{main-c_0}
The set of functions $f\in \mathcal{H}^\infty(B_{c_0})$ whose cluster set at every $x\in B_{c_0}$  contains a disk is strongly  $\mathfrak{c}$-algebrable and contains {(up to the zero function)} an almost isometric copy of $\ell_1$. In particular, it is spaceable.
\end{thm}
\medskip
We refer to \cite{Dis} for background on Banach spaces and to \cite{dineen,Mujica} for background on complex analysis on infinite dimensional spaces.

\bigskip

\section{Proofs of Theorems \ref{main-findim} and \ref{main-infdim}}

We begin by recalling the definitions of lineability, algebrability and spaceability.

\begin{dfn}
	 Let $X$ be a topological vector space and $\kappa$ be a cardinal number. A subset $M \subset X$ is said to be {\it $\kappa$-lineable} (resp. {\it $\kappa$-spaceable}) if there is a $\kappa$-dimensional vector subspace (resp. $\kappa$-dimensional closed vector subspace) $Y$ of $X$ such that $Y \subset M \cup \{0\}$.
\end{dfn}
Note that the definition of lineable sets makes sense even if $X$ is only a vector space without any topology. Recall that if $\mathcal{A}$ is a complex commutative algebra, a subset $G = \{x_i : i \in I\}$ of $\mathcal{A}$ is said to be {\it algebraically independent} whenever the following holds: given $n\in \mathbb N$, if  $Q \in \mathbb{C}[z_1, \ldots, z_n]$ is a polynomial such that  $Q(x_{i_1}, \ldots, x_{i_n})=0$ for some $x_{i_1}, \ldots, x_{i_n} \in G$ with different indexes  $i_1, \ldots, i_n \in I$, then $Q$ must be 0.

\begin{dfn}
 Let $X$ be an arbitrary set, $\mathcal{A}$ be an algebra of functions $f \colon X \longrightarrow \mathbb{C}$ and $\kappa$ be a cardinal number. A subset $M \subset \mathcal{A}$ is said to be {\it strongly $\kappa$-algebrable} if there is a sub-algebra $\mathcal{B}$ of $\mathcal{A}$ which is generated by an infinite algebraically independent set of generators with cardinality $\kappa$ and such that $\mathcal{B} \subset M \cup \{ 0 \}$.
\end{dfn}

Given $x \in  \overline{B}^{**}$ and $M \subset  \overline{B}^{**}$, we consider the following set:
$${\mathcal{F}}_M(B) := \{f \in \mathcal{H}^\infty(B) : \cap_{x \in M}Cl(f,x) \mbox{ contains a disc centered at } 0\}.$$
With this notation, Theorems~\ref{main-findim} and \ref{main-infdim}  can be restated as follows.

\begin{thm}
 	For $E=(\mathbb C^n,\|\,\|)$, the set $\mathcal{F}_{S}(B)$ is strongly $\mathfrak{c}$-algebrable.
\end{thm}

\begin{thm}
If $E$ is a separable infinite-dimensional Banach space, then $\mathcal{F}_{\overline{B}^{**}}(B)$ is strongly $\mathfrak{c}$-algebrable.
\end{thm}

Before proving these theorems, we recall that each function $f \in \mathcal{H}^\infty(B)$ can be extended to a function $\tilde{f} \in \mathcal{H}^\infty(B^{**})$ by means of the Aron-Berner extension. This extension is a multiplicative linear isometry (see \cite{DavGam}).
With this, we have the following characterization of the cluster set given in  \cite[p. 2357]{Aron1}:
\begin{eqnarray} \label{EqCl}
 Cl(f,z) = \{\mu \in \mathbb{C} :  \mbox{there exists a net } (x_\alpha) \subset B^{**},~ x_\alpha\overset{w(E^{**},E^*)}{\longrightarrow} z \mbox{ and } \tilde{f}(x_\alpha) \rightarrow \mu \}
\end{eqnarray}
for each $f \in \mathcal{H}^\infty(B)$.

We also recall that a sequence $(x_k^{**})$ in $B^{**}$ is an {\it interpolating sequence} for $\mathcal{H}^\infty(B)$ if for any sequence $(\lambda_k) \in \ell_\infty$, there is ${f} \in \mathcal{H}^\infty(B)$ such that $\tilde{f}(x_k) = \lambda_k$. The following theorem from \cite{GalMir} gives a sufficient condition for $(x_k^{**}) \subset B^{**}$ to be interpolating for $\mathcal{H}^\infty(B)$.

\begin{thm} \cite[Corollary 8]{GalMir} \label{thmIterpolating}
	Let $(x_k^{**})$ be a sequence in $B^{**}$ and $0 < c < 1$ such that $(1-\|x_{k+1}^{**}\|)/(1-\|x_k^{**}\|) < c$.
	Then $(x_k^{**})$ is an interpolating sequence for $\mathcal{H}^\infty(B)$.
\end{thm}

Now we set some notations. Given $N$, let $\alpha = (\alpha_1, \ldots, \alpha_N)$ be a $N$-tuple of nonnegative integers. As usual we denote by $supp(\alpha) := \{i : \alpha_i \not= 0\}$ the support of $\alpha$. And for each variable $X = (X_1, \ldots, X_N)$ we set $X^\alpha := X_1^{\alpha_1} \cdots  X_N^{\alpha_N}$ with the usual convention $X_i^0 = 1$. With this notation we may write any polynomial $Q \in \mathbb C[X_1, \ldots, X_N]$ as
\begin{eqnarray*} 
	Q(X) = \sum_{\alpha \in \Lambda} c_\alpha X^\alpha,
\end{eqnarray*}
where $\Lambda \subset \mathbb N_0^N$ is a finite set and $c_\alpha \in \mathbb C\setminus \{0\}$ for each $\alpha\in \Lambda$.

The proof of Theorems~\ref{main-findim} and \ref{main-infdim} will be splitted into two lemmas.

\begin{lem} \label{lem3}
\begin{enumerate} [label=(\roman*)]
\item \label{lem3-1} Let $E=(\mathbb C^n,\|\,\|)$ and
 $M \subset S$ be a countable set. Then the set $\mathcal{F}_M(B)$ is strongly $\mathfrak{c}$-algebrable.

\item \label{lem3-2}
	Let $E$ be an infinite-dimensional complex Banach space, and let $M \subset \overline{B}^{**}$ be a countable set. Then the set $\mathcal{F}_M(B)$ is strongly $\mathfrak{c}$-algebrable.
\end{enumerate}
\end{lem}

\begin{proof}
Let's only prove $\ref{lem3-2}$ since the proof of $\ref{lem3-1}$ is analogous. Set $M = \{x_k^{**} : k \in \mathbb{N}\}$.  By the variant of Josefson-Nissenzweig theorem presented in \cite[Ch.~XII,~Ex.~2]{Dis},  for each $p \in \mathbb{N}$ there exists a sequence $(v_{p,k}^{**})_k \subset S_{E^{**}}$ converging weak-star to $x_p^{**}$.

Let $\{\Theta_p : p \in \mathbb N\}$ be a partition of the natural numbers into (disjoint) infinite subsets. We can rearrange the family $(v_{p,k}^{**})_{p,k}$ in a single sequence $(w_k^{**})_k$ so that $x_p^{**}$ is the weak-star limit point of $\{w_k^{**}: k \in \Theta_p\}$ for each $p\in \mathbb N$.
Now we set $$z_k^{**} := \Big(1 - \frac 1 {2^k}\Big) w_k^{**}$$ for each $k \in \mathbb{N}$. It follows from Theorem \ref{thmIterpolating} that the sequence $(z_k^{**}) \subset B^{**}$ is an interpolating sequence for $\mathcal H^{\infty}(B)$. Also,  $x_p^{**}$ is the weak-star limit point of $\{z_k^{**} : k \in \Theta_p\}$ for each $p \in \mathbb N$.
We fix a bijection mapping
\begin{equation*}
n \in \mathbb N \mapsto s(n) = (s_1(n),s_2(n)) \in \mathbb Q_+^* \times \mathbb N.
\end{equation*}
We write $\Theta_p := \{k_1^p < k_2^p < \cdots\}$ and define
$$x_{r,N}^p= z^{**}_{k^p_{s^{-1}(r,N)}} \quad \text{ for } r\in   \mathbb Q_+^* \text{ and } N\in\mathbb N.
$$
Moreover, for each $k \in \mathbb N$ and $\xi \in (0,1)$ we take $k-1 = t_{\xi,0}^k < t_{\xi,1}^k < \cdots < t_{\xi,N_\xi^k}^k = k$ a partition of the interval $[k-1,k]$ such that $|t_{\xi,i}^k - t_{\xi,i-1}^k| < \xi^k$ and $N_\xi^k \in \mathbb N$. We set $I_{\xi,i}^k := [t_{\xi,i-1}^k,t_{\xi,i}^k)$ for $\xi \in (0,1), k \in \mathbb N$ and $i = 1,2,\ldots, N_\xi^k$.

Let $(a_k)_k$ be a sequence in $\mathbb{D}$ whose set of accumulation points is $\overline{\mathbb{D}}$. Since $(z_k^{**})$ is interpolating for $\mathcal H^\infty(B)$, for each $\xi \in (0,1)$ there is $f_\xi \in \mathcal H^\infty(B)$ so that
\begin{eqnarray} \label{DefFuncAlg}
\tilde{f_\xi}(x^p_{r,N}) :=
\left\{
\begin{array}{lll}
a_{N} & \mbox{if $r \in I_{\xi,i}^k$ and $i \equiv 2$ (mod 3);}\\
1 & \mbox{if $r \in I_{\xi,i}^k$ and $i \equiv 1$ (mod 3);}\\
0 & \mbox{otherwise.}
\end{array}
\right.
\end{eqnarray}
It is clear that for $\xi_1 < \xi_2 < \cdots < \xi_m$ in $(0,1)$ there is $k_0 \in \mathbb N$ such that $N_{\xi_m}^{k_0} \geq 3$ and $(\xi_{j+1} / \xi_j)^{k_0} > 3$ for each $j = 1,2, \ldots, m-1$. So, for every $j = 1, \ldots, m-1$ and $i = 1, \ldots, N_{\xi_{j+1}}^{k_0}$, the interval $I_{\xi_{j+1},i}^{k_0}$ contains at least four consecutive points of the partition $\{t_{\xi_j,0}^{k_0}, \ldots, t_{\xi_j,N_{\xi_j}^{k_0}}^{k_0}\}$ of $[k_0-1,k_0]$. Set $C_0^N :=0$, $C_1^N := 1$ and $C_2^N := a_N$ for each $N \in \mathbb N$. It follows from (\ref{DefFuncAlg}) that for each $(i_1, \ldots, i_m) \in \{0,1,2\}^m$ there is $r \in [k_0-1,k_0) \cap \mathbb Q$ such that
\begin{equation}\label{eq-lasf}
(f_{\xi_1}(x^p_{r,N}), \ldots, f_{\xi_m}(x^p_{r,N})) = (C_{i_1}^N, \ldots, C_{i_m}^N)
\end{equation}
for every $p,N \in \mathbb N$.

To complete the proof we need to show that the set $\{f_\xi : \xi \in (0,1)\}$ is an algebraically independent set and that the sub-algebra $\mathcal B \subset \mathcal H^\infty(B)$ generated by it lies in $\mathcal F_M(B) \cup \{0\}$. But this follows from the following assertion:
\begin{enumerate} [label = ($\ast$)]
	\sloppy \item \label{Prob ast} For every $m\in \mathbb N$, each choice of numbers $\xi_1 < \cdots < \xi_m$ in $(0,1)$ and every polynomial $Q \in \mathbb C[X_1, \ldots, X_m] \setminus \{0\}$ without constant term, there is $\delta > 0$ such that $D_\delta(0) \subset Cl(Q(f_{\xi_1}, \ldots, f_{\xi_m}), x_p^{**})$ for every $p$.
\end{enumerate}
It is clear that Assertion \ref{Prob ast} implies $\mathcal{B} \subset \mathcal F_M(B) \cup \{0\}$. Moreover, it also implies that the set $\{f_\xi : \xi \in (0,1)\}$ is algebraically independent:  if $Q$ is any polynomial, we apply Assertion \ref{Prob ast} to $\tilde Q= Q-Q(0)$ to get that if  $Q$ is not constant, then $Q(f_{\xi_1}, \ldots, f_{\xi_m})$ is not a constant function (the closure of its image contains a disk).

Let us  prove Assertion \ref{Prob ast}. Take $m\in \mathbb N$, $\xi_1 < \cdots < \xi_m$ in $(0,1)$ and  $Q$ as in the statement. We write $Q$ as $$Q(X) = \sum_{\alpha \in \Lambda} c_\alpha X^\alpha $$with $c_\alpha\ne 0$ for $\alpha \in \Lambda$.
Take $\tilde{\alpha} \in \Lambda$ such that $\# supp(\tilde{\alpha}) \le \# supp(\alpha)$ for each $\alpha \in \Lambda$. We consider the set $\Lambda_1 := \{\alpha \in \Lambda : supp(\alpha) = supp(\tilde{\alpha})\}$  with the following complete order: $(\alpha_1, \ldots, \alpha_m) < (\beta_1, \ldots, \beta_m)$ if there is $i_0 \in \{1,\ldots,m\}$ such that $\alpha_{i_0} < \beta_{i_0}$ and $\alpha_i = \beta_i$ for each $i \in \{i_0 +1, \ldots, m\}$  (this order can be seen as a variant of the lexicographical order).

We then write $\Lambda_1 = \{\beta^1 < \beta^2 < \cdots < \beta^l\}$ with $\beta^i = (\beta_1^i, \ldots, \beta_m^i)$ for each $i$. Note that $l= \#\Lambda_1 \le \#\Lambda$. If $l>1$,
since $\beta^l$ is the maximum for the introduced order, there exists $m_0 \in \{1, \ldots, N\}$ such that $\beta_{m_0}^{l-1} < \beta_{m_0}^l$ and $\beta_m^{l-1} = \beta_m^l$ for $m = m_0+1, \ldots, N$. In the case that $\Lambda_1$ contains just an element, we take any $m_0 \in supp(\tilde{\alpha})$.

We set ${i_{m_0}}= 2 $ and define $i_{\ell}=1$ if $\ell\in supp(\tilde{\alpha}) \setminus \{m_0\}$ and $i_{\ell}=0$ if $\ell \not\in supp(\tilde{\alpha})$. As was proved above there exist $k_0 \in \mathbb N$ and $r \in [k_0-1,k_0) \cap \mathbb Q$ satisfying \eqref{eq-lasf} for the previous $i_\ell$'s. Hence, for $\alpha\in \Lambda_1$ we have
$$ \big(\tilde{f}_{\xi_1}(x^p_{r,N}),\ldots, \tilde{f}_{\xi_m}(x^p_{r,N})\big)^\alpha = \big( \underbrace{\tilde f_{\xi_{m_0}}(x^p_{r,N})}_{\,\, =a_N}\big)^{\alpha_{m_0}} \times \underbrace{\cdots\cdots\cdots}_{\text{powers of 1}} = (a_N)^{\alpha_{m_0}}. $$ Note that every $\alpha \in \Lambda_1$ satisfies $\alpha_{m_0}>0$ and, by the maximality of $\beta^l$,  we have $\alpha_{m_0}<\beta^l_{m_0}$ if $\alpha\ne\beta^l$.
On the other hand, for  $\alpha\not\in \Lambda_1$ we can take $\ell\in supp(\alpha) \setminus \Lambda_1$ (since $\Lambda_1$ has minimum cardinality among the supports of the $\alpha$'s in $\Lambda$). Then,
$$ \big(\tilde{f}_{\xi_1}(x^p_{r,N}),\ldots, \tilde{f}_{\xi_m}(x^p_{r,N})\big)^\alpha = \big(\underbrace{f_{\xi_\ell}(x^p_{r,N})}_{=0} \big)^{\alpha_\ell}\times \underbrace{\cdots\cdots\cdots}_{\text{other powers}}   = 0.$$
This means that $$ Q\big(\tilde{f}_{\xi_1}(x^p_{r,N}),\ldots, \tilde{f}_{\xi_m}(x^p_{r,N})\big)$$ turns out to be a linear combination of powers of $a_N$ (i.e., a polynomial in $a_N$). There can be some cancellation, but the maximal exponent $\beta^l_{m_0}$ only appears once, so it remains. We then have a (one variable) polynomial $q:\mathbb C\to\mathbb C$ of degree $d=\beta^l_{m_0}>0$ such that
$$ Q\big(\tilde{f}_{\xi_1}(x^p_{r,N}),\ldots, \tilde{f}_{\xi_m}(x^p_{r,N})\big) = q(a_N).$$
Note that the  polynomial $q$ is independent of $p$ and that $q(0)=0$. The open mapping theorem (applied to $q$) gives a disc $D_\delta(0)$ in $q(\mathbb D)$.
Since  $x^p_{r,N}$ converges to $ x_p^{**}$ in the weak-star topology and $(a_N)_N$ is dense in $\overline {\mathbb D}$, it follows from (\ref{EqCl}) that this disc must be contained in $Cl(Q(f_{\xi_1}, \ldots, f_{\xi_m}), x_p^{**})$. This proves  Assertion~\ref{Prob ast} and completes the proof of the lemma.
\end{proof}

\begin{lem} \label{lem4}
\begin{enumerate} [label=(\roman*)]
\item\label{lem4-1} For $E=(\mathbb C^n,\|\,\|)$ and $M \subset S$ we have	$\mathcal{F}_M(B) = \mathcal{F}_{\overline{M}}(B)$.
\item\label{lem4-2} For $E$  an infinite-dimensional Banach space and $M \subset \overline{B}^{**}$ we have $\mathcal{F}_M(B) = \mathcal{F}_{\overline{M}^{w^*}}(B)$.
\end{enumerate}
\end{lem}

\begin{proof}
We only prove $\ref{lem4-2}$  since the proof of $\ref{lem4-1}$	folows in the same way. Let us see that $\bigcap_{x \in M} Cl(f,x) = \bigcap_{x \in \overline{M}^{w^*}} Cl(f,x)$. One inclusion is clear. For the other one, take $\mu \in \mathbb{C}$  such that $\mu \not\in \bigcap_{x \in \overline{M}^{w^*}} Cl(f,x)$. Then, there exists $x_0 \in \overline{M}^{w^*}$ such that $\mu \not\in Cl(f,x_0)$. That is, there is a weak-star open subset of $E^{**}$ such that $x_0 \in U$ and $\mu \not\in \overline{f(U \cap B)}$. Since $x_0 $ lies in $ \overline{M}^{w^*}$, we can find $y_0 \in M \cap U$. Therefore $U$ is a weak-star neighborhood of $y_0$ and  $\mu$ does not belong to $\overline{f(U \cap B)}$, which implies that $\mu \not\in Cl(f,y_0) \supset \bigcap_{y \in M} Cl(f,y)$. This completes the proof.
\end{proof}

\medskip

The previous two lemmas easily give the $\mathfrak{c}$-algebrability of Theorems \ref{main-findim} and \ref{main-infdim}.

\begin{proof}[Proof of Theorems \ref{main-findim} and \ref{main-infdim}]
 The finite dimensional case is immediate. For the infinite dimensional case, we take a countable set $M$  satisfying $\overline{M} = \overline{B}$, use Goldstine theorem to conclude that ${\overline{M}}^{w^*} = \overline{B}^{**}$ and then use the lemmas.
\end{proof}

\section{Proofs of Theorems  \ref{main-ell_p} and \ref{main-c_0}} \label{sec-elep}

In this section we give the proofs of Theorems \ref{main-ell_p} and \ref{main-c_0}. First, we show that a general result like Theorem~\ref{main-infdim} does not hold for $A_u(B_{E})$, for example,  whenever $E$ is uniformly convex or $E=\ell_1$.

We denote by
$A(B)$ the algebra of uniform limits on $B$ of polynomials in the
functions in $E^*$. The following is a consequence of \cite[Lemma 2.4]{Aron4}: take $x \in \overline B$ and suppose there exists $g\in A(B)$
such that $g(x)=1$, while $|g|$ is bounded by a constant strictly
less than $1$ on any subset of $B$ at a positive distance from $x$.
Then, if  $f\in \mathcal{H}^\infty(B)$ satisfies that
 $f(y) \to \lambda$ whenever $y\in B$ tends to $x$ in
norm, then  $Cl(f,x)=\{\lambda\}$.

\begin{pps}\label{notinAu}Let $E$ be either a uniformly convex Banach space or $\ell_1$
and let $f\in \mathcal{H}^\infty(B)$. If there exists a  fixed disk contained in $Cl(f,z)$ for every $z\in {B}$, then $f$ does not belong to $A_u(B)$.
\end{pps}
\begin{proof}
Proposition 4.1 from \cite{Far98}  shows that if $E$ is a
uniformly convex Banach space, then for any $x\in S$ there exists
a function $g$ satisfying the hypotheses of \cite[Lemma 2.4]{Aron4} mentioned above.
In \cite[Theorem 2.6]{AcoLou07} it is proved that the same holds for any point of the unit sphere of $\ell_1$. Hence, in either case we have that  $Cl(h,z)$ is a singleton for every $h\in A_u(B)$ and every $z\in S$.

Now, if there exists a fixed disk contained in $Cl(f,z)$ for every $z\in {B}$, by Lemma~\ref{lem4} the same holds for every $z\in \overline{B^{**}}$ and, in particular, for $z\in S$. By the previous considerations, this means that $f$ cannot belong to $A_u(B)$.
\end{proof}

We recall that a continuous $m$-homogeneous polynomial from $E$ into $\mathbb C$ is the restriction to the diagonal of a continuous $m$-linear functional. The symbol $\mathcal{P}(^mE)$ denotes the vector space of all continuous $m$-homogeneous polynomials from $E$ into $\mathbb{C}$. For every $m \in \mathbb N$ and $M \subset \overline{B}^{**}$ we set
$$\mathcal{E}_M(B) := \{f \in \mathcal H^\infty(B) : Cl(f,x) \mbox{ contains a disk for each } x \in M\};$$
$$\mathcal{E}_M^m(B) := \mathcal{E}_M(B) \cap \mathcal{P}(^mE).$$

\begin{rem}
  A function $f$ in $A_u(B)$ can be continuously extended to the closure $\overline B$ (extension which we still denote by $f$). It is easy to see that, in this case, the cluster set of $f$ at $z\in \overline B^{**}$ satisfies
$$
Cl(f,z) = \{\mu \in \mathbb{C} :  \mbox{ there exists a net } (x_\alpha) \subset \overline B,~ x_\alpha\overset{w(E^{**},E^*)}{\longrightarrow} z \mbox{ and } {f}(x_\alpha) \rightarrow \mu \}.
$$
As a consequence, to show that a cluster set contains a disk we can also consider vectors lying in the unit sphere.
\end{rem}

Following \cite[Definition~2.39]{dineen}, we say that a sequence $(x_n)\subset E$ has a lower $q$-estimate, $1\le q<\infty$, if there exists a constant $C>0$ such that for any sequence of scalars $(a_n)$ we have \begin{equation}\label{eq-lowerp} \sum_{n=1}^{k}|a_n|^q \le C \big\| \sum_{n=1}^{k} a_n\,x_n \big\|^q\end{equation} for any $k\in \mathbb N$.

Clearly, the canonical basis of $\ell_p$ has a lower $p$-estimate with constant $C=1$. Let us recall now the definition of Lorentz
spaces; further details and properties can be found in
\cite[Section 4.e]{LibroLiTz1} and \cite[Section
2.a]{LibroLiTz2}.\ Let $(w_k)_{k=1}^{\infty}$ be a decreasing
sequence of positive numbers such that $w_1=1$, $\lim_{k} w_k=0$
and $\sum_{k=1}^{\infty} w_k=\infty$ and let $1\leq p<\infty$.\
The corresponding Lorentz sequence space, denoted by $d(w,p)$
is defined as the space of all sequences $(x_k)_{k}$ such that
\[
\| x \| = \sup_{\pi \in \Sigma_{\mathbb{N}}} \bigg(
\sum_{k=1}^{\infty} |x_{\pi(k)}|^{p} w_k \bigg)^{1/p}<\infty,\]
where $\Sigma_{\mathbb{N}}$ denotes the group of permutations of the natural numbers. It is clear from the definition that we have a norm one inclusion $$\ell_p\overset{1}{\hookrightarrow } d(w,p).$$ For $1<p<\infty$ the space $d(w,p)$ is reflexive and its dual is also a Banach sequence space. Therefore, by duality we have
$$d(w,p)^* \overset{1}{\hookrightarrow } \ell_{p'},$$ which means that the canonical basis of $d(w,p)^*$ has lower $p'$-estimate with constant $C=1$.

Suppose now that $w$ belongs to $\ell_r$ for some $1<r<\infty$. In this case, we have the norm one inclusion
$$\ell_r\overset{1}{\hookrightarrow } d(w,1).$$  Let $d_*(w,1)$ denote the predual of $d(w,1)$. By duality, the previous inclusion gives
$$d_*(w,p) \overset{1}{\hookrightarrow } \ell_{r'},$$ and then the canonical basis of $d_*(w,p)$ has a lower $r'$-estimate with constant $C=1$.

Since the canonical basis of $\ell_p$ ($1<p<\infty$), $d(w,p)^*$ ($1<p<\infty$) and $d_*(w,1)$ are weakly null, the following proposition shows the spaceability part of Theorem~\ref{main-ell_p} except for the case $E=\ell_1$.

\begin{pps}
  Suppose $E$ has a normalized weakly-null Schauder basis $(x_n)$  with lower $q$-estimate. Then, for any $m \ge q$ the set $\mathcal{E}_B^m(B) \cup \{0\}$ contains an isomorphic copy of  $\ell_\infty$ (which is isometric if the constant $C$ in \eqref{eq-lowerp} is one).
\end{pps}
\begin{proof} Let $(a_j)_j$ be a sequence in $\mathbb{D}$ whose set of  accumulation points is $\overline{\mathbb{D}}$ and take $m\ge q$. Given an infinite subset $\Theta\subset \mathbb N$, we write $\Theta = \{n_1 < n_2 < \cdots\}$. We define $f_\Theta:E\to \mathbb C$ as
\begin{equation}\label{eq_ftheta}
  f_\Theta\bigg(\sum_n b_n x_n\bigg) =  \sum_{j=1}^\infty a_j (b_{n_j})^m.
\end{equation} To see that $f_\Theta$ is well defined, we just note that $$\sum_{j=1}^\infty |a_j| |b_{n_j}|^m \leq \sum_{n=1}^\infty |b_n|^m \le \big(\sum_{n=1}^\infty |b_n|^q\big)^{m/q} \le C \big\|\sum_{n=1}^\infty b_n x_n\big\|^m.$$
This also shows that  $f_\Theta$ is a continuous $m$-homogeneous polynomial.

Let
  $\{\Theta_k : k \in \mathbb{N}\}$ be a family of pairwise disjoint infinite subsets of $\mathbb{N}$. We define a mapping $\beta\in \ell_\infty \mapsto F_\beta \in \mathcal{P}(^mE)$ by $$F_\beta(x)=\sum_{k=1}^\infty \beta_k f_{\Theta_k}(x)$$ and $f_{\Theta_k}$ is as in \eqref{eq_ftheta}. To see that $F_\beta$ is a well defined and continuous $m$-homogeneous  polynomial, we write $\Theta_k = \{n^k_1 < n^k_2 < \cdots\}$. We have
  $$\sum_{k=1}^\infty |\beta_k|  \sum_{j=1}^\infty  |a_j| \, |b_{n^k_j}|^m \le \|\beta\|_\infty \sum_{k=1}^\infty  \sum_{j=1}^\infty   \, |b_{n^k_j}|^m \le \|\beta\|_\infty \bigg(\sum_{n=1}^{\infty}|b_n|^p\bigg)^{m/p} \le C \|\beta\|_\infty \bigg\| \sum_{n=1}^{\infty} b_n\,x_n \bigg\|^m.$$ This also shows that $\|F_\beta\|\le C\|\beta\|_\infty$ and the double series is absolutely convergent.

  Now, for $y = \sum_{n=1}^\infty c_n x_n \in B$, we note  that
  \begin{eqnarray*}
  F_\beta(y + (1-\|y\|) x_{n_j^k}) & = & \beta_k (c_{n_j^k} + 1 - \|y\|)^m a_j + \sum_{\substack{r,\ell = 1 \\ (r,\ell) \not= (k,j)}}^\infty \beta_r (c_{n_\ell^r})^m a_\ell  \\ &=&
  \beta_k (c_{n_j^k} + 1 - \|y\|)^m a_j + F_{\beta}(y) -  \beta_k (c_{n_j^k})^m a_j.
  \end{eqnarray*}
 Recall that $(a_j)_j$ accumulates at every element of the closed unit disk and note that $c_{n_j^k}\to 0$ as $j\to \infty$. Therefore, taking $y=0$ we have $\|F_\beta\|\ge |\beta_k|$ for all $k$, which means that $\|F_\beta\|\ge \|\beta\|_\infty$. On the other hand, if $\beta  \ne 0$, this shows that the cluster set of $F_\beta$ at $y$ contains a disk centered
  at $F_\beta(y)$ of radius $\|\beta\|_\infty(1-\|y\|)^m$.
\end{proof}

Note that for $E=\ell_p$ with the canonical basis and $\Theta=\mathbb N$, the function $f_\Theta$ in \eqref{eq_ftheta} is the function given in the Example after \cite[Lemma 2.1]{Aron4}.
As we mentioned above, the canonical bases of $\ell_p$ and $d_*(w,p)$ have lower $p$-estimates with constant $C=1$  and are weakly null for $1<p<+\infty$. Since $\mathcal{P}(^mE)$ is contained in $A_u(B)$, we conclude that for these spaces the set of functions $f\in A_u(B)$, whose cluster set at every $y \in B$  contains a disk is $\mathfrak{c}$-spaceable.

The canonical basis of  $\ell_1$ is not weakly null, so the previous proposition does not apply. For the following proposition, that settles the case $E=\ell_1$, we adapt an example presented in the proof of  \cite[Proposition 3.2]{Aron1}.

\begin{pps} \label{pps_eleuno}
  For any even $m$, the set $\mathcal{E}_{B_{\ell_1}}^m(B_{\ell_1})\cup \{0\}$ contains an isometric copy of~$\ell_\infty$.
\end{pps}

\begin{proof} We take again a sequence $(a_j)_j$ in $\mathbb{D}$ whose set of  accumulation points is $\overline{\mathbb{D}}$ and a family $\{\Theta_k : k \in \mathbb{N}\}$ of pairwise disjoint infinite subsets of $\mathbb{N}$. For each $k \in \mathbb N$, we let $\{I_r^{k}\}_{r \in \mathbb{N}}$ be a countable partition of $\Theta_k$ with each $I_r^k$ infinite. We define
\begin{equation}\label{eq_ftheta_eleuno}
  f_{\Theta_k} : (x_n) \in \ell_1 \longmapsto \sum_{r=1}^\infty {a_r} \sum_{n \in I_r^k} x_n^m \in \mathbb{C}.
\end{equation}
Let us observe that the linear mapping  $\beta\in \ell_\infty \mapsto F_\beta \in \mathcal{P}(^m\ell_1)$  given by $$F_\beta(x)= \sum_{k=1}^\infty \beta_k f_{\Theta_k}(x)$$ is an isometry. Indeed,
$$\sum_{k=1}^\infty |\beta_k| \sum_{r=1}^\infty |a_r| \sum_{n \in I_r^k} |x_n|^m \le \|\beta\|_\infty \bigg(\sum_{n=1}^\infty |x_n|\bigg)^m = \|\beta\|_\infty \|(x_n)\|_1^m,$$
and hence $\|F_\beta\| \le \|\beta\|_\infty$. On the other hand, if $s \in I_r^k$ then $F_\beta(e_s) = \beta_k a_r$. By the density of $(a_j)_j$ in $\overline{\mathbb{D}}$, this implies $\|F_\beta \| \ge |\beta_k|$ for each $k$, which means that $\|F_\beta\| \ge \|\beta\|_\infty$.

Now note that
\begin{equation}\label{eq_clausura}
0 \in \overline{\bigg\{\frac{e_s-e_t}2\colon\ s > t \ge N\, ,\, s,t\in I_r^k \bigg\}}^{w(\ell_1,\ell_\infty)}
\end{equation}
for every $N,r,k \in \mathbb N$. Note also that for $s,t \in I_r^k$, $s \not= t$, and $y = (y_n) \in B_{\ell_1}$, we have
\begin{eqnarray*}
	 F_\beta \left(y + (1-{\|y\|}_1) \left(\frac{e_s - e_t}{2}\right)\right) &=& \sum_{j=1}^\infty \beta_j \sum_{i=1}^\infty a_i \sum_{n \in I_i^j} \left[y_n + (1-{\|y\|}_1) \left(\frac{\delta_{s,n} - \delta_{t,n}}{2}\right)\right]^m \\  &&\hspace{-2.3in} = \sum_{j=1}^\infty \beta_j \sum_{i=1}^\infty a_i \sum_{\substack{n \in I_i^j \\ n \not= s, n \not= t}} y_n^m + \beta_k a_r \left[ \left(y_s + \frac{1 - {\|y\|}_1}{2}\right)^m + \left(y_t - \frac{1 - {\|y\|}_1}{2}\right)^m \right] \\ &&\hspace{-2.3in} = F_\beta(y) - \beta_k a_r (y_s^m + y_t^m) + \beta_k a_r \left[ \left(y_s + \frac{1 - {\|y\|}_1}{2}\right)^m + \left(y_t - \frac{1 - {\|y\|}_1}{2}\right)^m \right].
\end{eqnarray*}
From this and \eqref{eq_clausura} we deduce that $F_\beta(y) + 2^{1-m} \beta_k a_r (1 - {\|y\|}_1)^m$ belongs to $Cl(F_\beta,y)$ for every $r,k \in \mathbb N$. Since $y_s\to 0$, $(a_j)_j$ accumulates in $\overline{\mathbb{D}}$ and $Cl(F_\beta,y)$ is compact, we conclude that $Cl(F_\beta,y)$, $\beta \not= 0$, contains the disk centered at $F_\beta(y)$ and radius $2^{1-m} \|\beta\|_\infty(1- {\|y\|}_1)^m$.
\end{proof}

Again, since $\mathcal{P}(^m\ell_1)$ is contained in $A_u(B_{\ell_1})$, the set of functions $f\in A_u(B_{\ell_1})$ whose cluster set at every $y \in B_{\ell_1}$ contains a disk is $\mathfrak{c}$-spaceable.

\medskip
We now focus on  Theorem~\ref{main-c_0}. For this, we adapt an example presented in \cite{Aron4} (see the example there before Theorem 5.1).

\begin{pps}
  The set $\mathcal{E}_{B_{c_0}}(B_{c_0})\cup \{0\}$ contains an almost isometric copy of  $\ell_1$.
\end{pps}
\begin{proof}
Take $r_n < 1$ increasing rapidly to $1$ and let $\delta := \prod r_n$. For each infinite subset $\Theta =\{n_1 < n_2 < n_3 < \cdots\} \subset \mathbb{N}$ we define $f_\Theta : B_{c_0} \longrightarrow \mathbb{C}$ by
\begin{equation}\label{eq_ftheta_c0}
 f_\Theta(x) = \delta^{-1}\prod_{j\in\mathbb N} (r_j - x_{n_j})/(1 - r_j x_{n_j}).
\end{equation}
It is clear that $\|f_\Theta\| \le \delta^{-1}$. We now choose a family $\{\Theta_k : k \in \mathbb{N}\}$ of pairwise disjoint infinite subsets of $\mathbb{N}$ and, for each $k$, we consider the corresponding $f_{\Theta_k}$. The linear mapping  $\beta\in \ell_1 \mapsto F_\beta \in \mathcal{H}^\infty(B_{c_0})$ given by $$F_\beta(x)=\sum_{k=1}^\infty \beta_k f_{\Theta_k}(x)$$ is clearly well defined and continuous (in fact, $\|F_\beta\|\le \delta^{-1}\|\beta\|_{1}$).

We denote $\Theta_k := \{n^k_1 < n^k_2 < \cdots\}$ for each $k \in \mathbb N$. Suppose $\beta \not= 0$ and fix $N\in\mathbb N $ and  $\mu_1, \dots,\mu_N \in \mathbb{D}$ so that $(\beta_1, \ldots, \beta_N) \not= 0$. For  $k=1,\dots,N$ and $j \in \mathbb N$ we take $\lambda_{n^k_j}\in\mathbb D$ such that $\mu_k = (r_j - \lambda_{n^k_j})/(1 - r_j\lambda_{n^k_j})$. Moreover, for $y = (y_n) \in B_{c_0}$ we set
$$z_j^N := \sum_{i=1}^N \left[(1-|y_{n_j^i}|) \lambda_{n_j^i} - y_{n_j^i}\right] e_{n_j^i}.$$ Note that $(z_j^N)_{j=1}^\infty \subset B_{c_0} \cap (-y + B_{c_0})$ is a weakly null sequence and
\begin{eqnarray*}
F_\beta \left(y+z_j^N\right) &=& \sum_{k=1}^N \beta_k \delta^{-1} \frac{r_j - (1 - |y_{n_j^k}|)\lambda_{n_j^k}}{1 - r_j(1 - |y_{n_j^k}|)\lambda_{n_j^k}} \prod_{\ell \not= n_j} \frac{r_\ell - y_{n_\ell^k}}{1-r_\ell y_{n_\ell^k}} + \sum_{k=N+1}^\infty \beta_k f_{\Theta_k}(y) \\ &\xrightarrow{j \to \infty}&  \sum_{k=1}^N \beta_k \delta^{-1} \mu_k \delta_k + \sum_{k=N+1}^\infty \beta_k f_{\Theta_k}(y),
\end{eqnarray*}
where $\delta_k := \displaystyle \prod_{\ell=1}^\infty \frac{r_\ell - y_{n_\ell^k}}{1-r_\ell y_{n_\ell^k}}$. This implies that the disk centered at $\sum_{N+1}^\infty \beta_k f_{\Theta_k}(y)$ of radius $\|(\beta_k \delta^{-1} \delta_k)_{k=1}^N\|_1$ is contained in the cluster set of $F_\beta$ at $y$. On the other hand, for $y=0$ we have $\delta_k = \delta$ for each $k$, hence one may take $N \to \infty$ and conclude that $\|\beta\|_1\le \|F_\beta\|$.  Since $\delta$ can be taken arbitrarily close to 1, this completes the proof.
\end{proof}

Now we prove the algebratility part of Theorem~\ref{main-ell_p}

\begin{pps} Let $E$ be either $\ell_p$ ($1\le p <\infty$) or $d(w,p)^*$ ($1< p <\infty$) or $d_*(w,1)$ with $w\in \ell_s$ for some $1<s<\infty$.
  The set of functions $f\in A_u(B)$  whose cluster set at every $y \in B$ contains a disk is strongly $\mathfrak{c}$-algebrable.

\end{pps}

\begin{proof}

Let $\{\Theta_i : i \in I\}$ be a family of infinite subsets of $\mathbb{N}$ such that $I$ has cardinality  $\mathfrak c$,  $\mathbb{N} = \bigcup_{i \in I} \Theta_i$ and such that, for  $i \not= j$, the set $\Theta_i \cap \Theta_j$ is finite. For each $i \in I$, let $g_i:=f_{\Theta_i}$, where $f_{\Theta_i}$ is the function defined in \eqref{eq_ftheta} for $E=\ell_p$, $1<p<\infty$, or \eqref{eq_ftheta_eleuno} for $E=\ell_1$.

Assume for a moment that  $G = \{g_i : i \in I\}$ satisfies the following condition (which we will prove later):
\begin{quotation}
 \noindent given $g_{i_1}, \ldots, g_{i_n}\in G$ with all different indices $i_1, \ldots, i_n \in I$, $y \in B_E$, and  $a_1, \ldots, a_n \in \mathbb{D}$, there exist a weakly null net $(w_\alpha)$ in $B_E$ and a constant $C = C(n,y) > 0$ such that \begin{equation}\label{eq_condition}\lim_{\alpha}g_{i_j}\left(y+ (1-\|y\|) w_\alpha \right) = g_{i_j}(y) + C\, a_j \quad j=1,\dots,n.\end{equation}
\end{quotation}

Let us see that $G$ is an algebraically independent set. Suppose that for some polynomial $Q \in \mathbb{C}[x_1, \ldots, x_n]$ we have $Q(g_{i_1}, \ldots, g_{i_n}) = 0$ for $g_{i_1}, \ldots, g_{i_n} \in G$ with distinct indices. For each $a := (a_1, \ldots, a_n) \in \mathbb{D}^n$, we take the net $(w_\alpha)$ in $B_E$ satisfying \eqref{eq_condition} with $y=0$. Hence,
	$$Q(Ca_1, \ldots, Ca_n) = \lim_{\alpha} Q(g_{i_1}(w_\alpha), \ldots, g_{i_n}(w_\alpha)) = 0,$$
	which implies $Q = 0$ on the open set $(C\mathbb{D})^n$, and therefore $Q = 0$ on $\mathbb{C}^n$.
	
Now we show that any nonzero $f$ in the algebra spanned by $G$ has a disk in its cluster set at every $y \in B_E$. Let $Q \in \mathbb{C}[x_1, \ldots, x_n] \setminus \{0\}$ be a polynomial without constants and let $f=Q(g_{i_1}, \ldots, g_{i_n})$ for all different $i_1, \ldots, i_n \in I$. If $y \in B_E$ is fixed and $C=C(n,y)$ is as above, we may consider the analytic function
	$$h_y : a = (a_1,\ldots,a_n) \in \mathbb{D}^n \mapsto Q(g_{i_1}(y) + C\, a_1, \ldots, g_{i_n}(y) + C\, a_n).$$
For each $a \in \mathbb{D}^n$ we use \eqref{eq_condition} again to get a weakly null net $(w_\alpha) \subset B_E$  such that
\begin{eqnarray*}
h_y(a) &=& \lim_{\alpha} Q(g_{i_1}(y + (1-\|y\|)w_\alpha), \ldots, g_{i_n}(y + (1-\|y\|)w_\alpha)) \\ &=& \lim_{\alpha} f(y + (1-\|y\|)w_\alpha).
\end{eqnarray*}
This means that the image of $h_y$ is contained in $Cl(f,y)$. It remains to notice that $h_y$ is not constantly 0 (this is the algebraic independence shown above) so there is a disk centered at $h_y(0)$ in the image of $h_y$ and, then, in the cluster set of $f$ at $y$.

\medskip
It remains to show \eqref{eq_condition}. We do it for $E=\ell_1$, the other cases being even simpler (the nets are actually sequences). We choose $m=2$ in \eqref{eq_ftheta_eleuno} (any even number works). For each $j=1,\dots,n$ we know (proof of Proposition \ref{pps_eleuno}) that there exists a net of the form $((e_{s^{j}_\alpha}-e_{t^{j}_\alpha})/2)_\alpha$ with all  $s^{j}_\alpha, t^{j}_\alpha\in \Theta_{i_j}$ satisfying
$$g_{i_j}\bigg(y + (1 - \|y\|_1)\frac {e_{s^{j}_\alpha}-e_{t^{j}_\alpha}}{2n}\bigg)=f_{\Theta_{i_j}} \bigg(y + (1 - \|y\|_1)\frac { e_{s^{j}_\alpha}-e_{t^{j}_\alpha}}{2n}\bigg) \to g_{i_j}(y) + \frac{(1-\|y\|_1)^2}{2n^2} \, a_j.$$
Note that we are using the same index set for the $n$ different nets. If $\alpha$ is large enough, all the involved elements will have disjoint supports. So, if we write $$w_\alpha =\frac 1 {2n}  \sum_{r=1}^{n} e_{s^{r}_\alpha}-e_{t^{r}_\alpha}, $$ we have for each $j=1,\dots,n$, $$g_{i_j}(y + (1 - \|y\|_1)w_\alpha) = \, g_{i_j}\bigg(y + (1 - \|y\|_1)\frac {e_{s^{j}_\alpha}-e_{t^{j}_\alpha}}{2n}\bigg) \to g_{i_j}(y) + \frac{(1-\|y\|_1)^2}{2n^2} \, a_j.$$ This gives \eqref{eq_condition} with $C=(1-\|y\|_1)^2/(2n^2)$.
\end{proof}

\subsection*{Acknowledgements}
The first author thanks the Departamento de Matem\'atica of Universidad de Buenos Aires and its members for their kind hospitality.

\bigskip
\noindent
\begin{tabular}{l}
	Thiago R. Alves\\
	Departamento de Matem\'{a}tica\\
	Instituto de Ci\^{e}ncias Exatas\\
	Universidade Federal do Amazonas\\
	69.077-000 -- Manaus -- Brazil\\
	e-mail: \tt{tralves.math@gmail.com}
\end{tabular}
\bigskip
\hspace{11mm}
\begin{tabular}{l}
	Daniel Carando\\
	Departamento de Matem\'atica\\
	Facultad de Cs. Exactas y Naturales\\
    Universidad de Buenos Aires\\
	and IMAS-UBA-CONICET, Argentina.\\
	e-mail: \tt{dcarando@dm.uba.ar}
\end{tabular}
\end{document}